\documentclass[11pt]{amsart}
\usepackage{lmodern}
\usepackage{amsmath, amsthm, amssymb, amsfonts}
\usepackage[normalem]{ulem}
\usepackage{hyperref}

\usepackage{verbatim} 
\usepackage{longtable}

\usepackage{mathtools}

\usepackage{tikz}
\usetikzlibrary{decorations.pathmorphing}
\tikzset{snake it/.style={decorate, decoration=snake}}

\usepackage{caption}

\usepackage{tikz-cd}
\usetikzlibrary{arrows}

\theoremstyle{plain}
\newtheorem{thm}{Theorem}[section]
\newtheorem{cor}[thm]{Corollary}
\newtheorem{lem}[thm]{Lemma}
\newtheorem{prop}[thm]{Proposition}

\newtheorem{question}[thm]{Question}

\theoremstyle{definition}

\theoremstyle{remark}
\newtheorem{rmk}[thm]{Remark}

\newcommand{\BA}{{\mathbb{A}}}

\newcommand{\BC}{{\mathbb{C}}}

\newcommand{\BF}{{\mathbb{F}}}

\newcommand{\BN}{{\mathbb{N}}}

\newcommand{\BP}{{\mathbb{P}}}
\newcommand{\BQ}{{\mathbb{Q}}}
\newcommand{\BR}{{\mathbb{R}}}

\newcommand{\BZ}{{\mathbb{Z}}}

\newcommand{\CE}{{\mathcal E}}
\newcommand{\CF}{{\mathcal F}}

\newcommand{\CH}{{\mathcal H}}

\newcommand{\CK}{{\mathcal K}}
\newcommand{\CL}{{\mathcal L}}
\newcommand{\CM}{{\mathcal M}}

\newcommand{\CO}{{\mathcal O}}

\newcommand{\CS}{{\mathcal S}}

\DeclareFontFamily{OT1}{rsfs}{}
\DeclareFontShape{OT1}{rsfs}{n}{it}{<-> rsfs10}{}
\DeclareMathAlphabet{\curly}{OT1}{rsfs}{n}{it}

\newcommand{\Coh}{\mathrm{Coh}}

\newcommand{\git}{\mathbin{
  \mathchoice{/\mkern-6mu/}% \displaystyle
    {/\mkern-6mu/}% \textstyle
    {/\mkern-5mu/}% \scriptstyle
    {/\mkern-5mu/}}}% \scriptscriptstyle

\usepackage{tikz}
\usepackage{lmodern}
\usetikzlibrary{decorations.pathmorphing}

\begin{document}
\title[On the P=W conjecture for $\mathrm{SL}_n$]{On the P=W conjecture for $SL_n$}
\date{\today}

\author[M.A. de Cataldo]{Mark Andrea de~Cataldo} 
\address{Stony Brook University}
\email{mark.decataldo@stonybrook.edu}

\author[D. Maulik]{Davesh Maulik}
\address{Massachusetts Institute of Technology}
\email{maulik@mit.edu}

\author[J. Shen]{Junliang Shen}
\address{Massachusetts Institute of Technology}
\email{jlshen@mit.edu}

\begin{abstract}
Let $p$ be a prime number. We prove that the $P=W$ conjecture for $\mathrm{SL}_p$ is equivalent to the $P=W$ conjecture for $\mathrm{GL}_p$. As a consequence, we verify the $P=W$ conjecture for genus 2 and $\mathrm{SL}_p$. For the proof, we compute the perverse filtration and the weight filtration for the variant cohomology associated with the $\mathrm{SL}_p$-Hitchin moduli space and the $\mathrm{SL}_p$-twisted character variety, relying on Gr\"ochenig--Wyss--Ziegler's recent proof of the topological mirror conjecture by Hausel--Thaddeus.

Finally we discuss obstructions of studying the cohomology of the $\mathrm{SL}_n$-Hitchin moduli space via compact hyper-K\"ahler manifolds. 
\end{abstract}

\maketitle

\setcounter{tocdepth}{1} 

\tableofcontents
\setcounter{section}{-1}

\section{Introduction}

Throughout the paper, we work over the complex numbers $\BC$. 

Let $C$ be a nonsingular projective curve of genus $g\geq 2$, and let $G$ be a reductive group. The P=W conjecture of de Cataldo, Hausel, and Migliorini \cite{dCHM1} predicts a surprising connection between the topology of $G$-Hitchin systems and the Hodge theory $G$-character varieties via the non-abelian Hodge correspondence. More precisely, it suggests that the perverse filtration for the Hitchin system associated with the $G$-Dolbeault moduli space $\CM_{\mathrm{Dol}}$ coincides with the weight filtration associated with the corresponding $G$-Betti moduli space $\CM_B$,
\begin{equation}\label{P=W}
``P=W": \quad P_k H^d(\CM_\mathrm{Dol}, \BQ) = W_{2k} H^d(\CM_B, \BQ)= W_{2k+1} H^d(\CM_B, \BQ), \quad \forall k,d\geq 0;
\end{equation}
see Section \ref{Section1} for a brief review.

When $G= \mathrm{GL}_n$, the $P=W$ conjecture was established for any genus $g$ and rank $n=2$ in \cite{dCHM1}, and very recently, for genus $g=2$ and arbitrary rank $n$ in \cite{dCMS}. Furthermore, for aribitrary genus and rank, \cite{dCMS} shows $P=W$ for the tautological generators of the cohomology, and reduces the full $P=W$ conjecture to the multiplicativity of the perverse filtration.

The $G=\mathrm{PGL}_n$ case is equivalent to the $\mathrm{GL}_n$ case for a fixed curve $C$; see \cite{dCMS} the paragraph following Theorem 0.2. It is natural to explore non-trivial examples of the $P=W$ phenomenon for a reductive group $G$ other than $\mathrm{GL}_n$ and $\mathrm{PGL}_n$. 

The purpose of this paper is to study $P=W$ for $G = \mathrm{SL}_n$. The case of $\mathrm{SL}_2$ was already established in \cite{dCHM1}. We provide in the following theorem an affirmative answer to the $P=W$ conjecture when the curve has genus $g=2$ and the rank $n$ is any prime number.

\begin{thm}\label{thm0.1}
The $P=W$ conjecture (\ref{P=W}) holds when $C$ has genus $g =2$ and $G = \mathrm{SL}_n$ with $n$ a prime number.
\end{thm}

We refer to Section \ref{Section1.5} for more precise statements. Here we briefly explain the main difference between the $G = \mathrm{GL}_n$ case and the $G = \mathrm{SL}_n$ case.

Let $\CM_{\mathrm{Dol}}$ be the $\mathrm{SL}_n$-Dolbeault moduli space assocated with a curve $C$ of genus $g \geq 2$ and a line bundle $L$ with $\mathrm{gcd}\left(c_1(L), n\right)=1$ (see Section \ref{Section1}). There is a natural action of the finite group $\Gamma = \mathrm{Pic}^0(C)[n]$ on $\CM_{\mathrm{Dol}}$ via tensor product. This group action yields a decomposition with respect to the irreducible characters of $\Gamma$, 
\begin{equation}\label{decomp}
    H^*(\CM_{\mathrm{Dol}}, \BQ) = H^*(\CM_{\mathrm{Dol}}, \BQ)^\Gamma \bigoplus H_{\mathrm{var}}^*(\CM_{\mathrm{Dol}}, \BQ).
\end{equation}
Here the $\Gamma$-invariant part $H^*(\CM_{\mathrm{Dol}}, \BQ)^\Gamma$ corresponds to the trivial character, and the variant cohomology $H_{\mathrm{var}}^*(\CM_{\mathrm{Dol}}, \BQ)$ corresponds to all the non-trivial characters. Note that we have the same decomposition (\ref{decomp}) for the Betti moduli space $\CM_B$. The $\Gamma$-invariant part 
\begin{equation}\label{invariant}
H^*(\CM_{\mathrm{Dol}}, \BQ)^\Gamma \subset H^*(\CM_{\mathrm{Dol}}, \BQ).
\end{equation}
is canonically identified with the cohomology of the corresponding moduli of stable $\mathrm{PGL}_n$-Higgs bundles (see (\ref{Gamma1})). In particular, it is the sub-vector space of $H^*(\CM_{\mathrm{Dol}}, \BQ)$ generated by the tautological classes with respect to a universal family.

As a consequence, $P=W$ for $\mathrm{GL}_n$ is equivalent to $P=W$ for the invariant part (\ref{invariant}). The following theorem proves $P=W$ for the variant cohomology for any genus when $n$ is prime.

\begin{thm}\label{thm0.2}
We have $P=W$ for the variant cohomology $H_{\mathrm{var}}^*(\CM_{\mathrm{Dol}}, \BQ)$ for any genus $g \geq 2$ with $n$ a prime number,
\[
P_k H_{\mathrm{var}}^d(\CM_\mathrm{Dol}, \BQ) = W_{2k} H_{\mathrm{var}}^d(\CM_B, \BQ), \quad \forall k,d\geq 0.
\]
\end{thm}

Theorem \ref{thm0.2} shows that, for a curve $C$ of genus $g \geq 2$, the $P=W$ conjecture for the groups $\mathrm{GL}_n$, $\mathrm{SL}_n$, and $\mathrm{PGL}_n$ are equivalent when $n$ is prime. The proof of Theorem \ref{thm0.2} relies on the recent proof \cite{GWZ} of the topological mirror conjecture \cite{HT}, and the calculations of $E$-polynomials for character varieties \cite{HRV, SL}.

In Section \ref{Section4}, we discuss obstructions of studying the cohomology of $\CM_{\mathrm{Dol}}$ via compact hyper--Ka\"ahler manifolds; see Propositions \ref{prop4.2} and \ref{prop4.3}. In particular, we provide obstructions to extend the method of \cite{dCMS} for proving the $P=W$ conjecture for genus $2$ and $\mathrm{GL}_n$ to the genus $2$ and $\mathrm{SL}_n$ case.

\subsection*{Acknowledgement}
We are grateful to Chen Wan and Zhiwei Yun for helpful discussions. 
The first-named author is partially supported by NSF DMS Grant 1901975.

\section{Hitchin moduli spaces and character varieties}\label{Section1}

Througout the section, we let $C$ be a nonsingular projective curve of genus $g \geq 2$. We also fix 2 integers $n,d$ satisfying $n \geq 2$ and $\mathrm{gcd}(n,d)=1$, and a line bundle $L \in \mathrm{Pic}^d(C)$.

\subsection{Moduli spaces} 
We review the two moduli spaces $\CM_{\mathrm{Dol}}$ and $\CM_{B}$ associated with the curve $C$, the group $\mathrm{SL}_n$, and the line bundle $L \in \mathrm{Pic}^d(C)$. We refer to \cite{dCHM1, Hit, Hit1, HLR, HRV} for more details.

The Dolbeault moduli space $\CM_{\mathrm{Dol}}$ parametrizes stable Higgs bundles 
\[
(\CE, \theta), \quad \theta: \CE \rightarrow \CE \otimes \Omega_C
\]
satisfying the conditions
\[
\mathrm{trace}(\theta) = 0, \quad \mathrm{det}(\CE) = L.
\]
The Hitchin system associated with $\CM_{\mathrm{Dol}}$ is a proper surjective morphism $\pi: \CM_{\mathrm{Dol}} \rightarrow \Lambda$ sending $(\CE, \theta)$ to the characteristic polynomial
\[
\mathrm{char}(\theta) \in \Lambda := \oplus_{i=2}^n H^0(C,\Omega_C^{\otimes i}).
\]
It is Lagrangian with respect to the canonical hyper-K\"ahler metric on $\CM_{\mathrm{Dol}}$. 
The Betti moduli space $\CM_B$ is the $\mathrm{SL}_n$-twisted character variety,
\begin{equation}\label{M_B}
\CM_B := \Big{\{}a_k, b_k \in \mathrm{SL}_n,~k=1,2,\dots,g: ~~\prod_{j=1}^g [a_j, b_j] = e^{\frac{2\pi \sqrt{-1} d}{n}}\mathrm{Id}_n \Big{\}}\git \mathrm{SL}_n,
\end{equation}
which is obtained as an affine $\mathrm{GIT}$ quotient with respect to the action by conjugation. 

Both $\CM_{\mathrm{Dol}}$ and $\CM_B$ are nonsingular quasi-projective varieties satisfying
\[
\mathrm{dim} (\CM_{\mathrm{Dol}}) = 2 \mathrm{dim}(\Lambda) = \mathrm{dim} (\CM_{B}) = (n^2-1)(2g-2).
\]
The non-abelian Hodge theory \cite{Simp, Si1994II} provides a differemorphism between $\CM_\mathrm{Dol}$ and $\CM_B$, which identifies the cohomology
\begin{equation}\label{NAH}
    H^*(\CM_{\mathrm{Dol}}, \BQ) = H^*(\CM_{B}, \BQ).
\end{equation}

\subsection{Perverse filtrations} The $P=W$ conjecture (\ref{P=W}) predicts the match of two completely different structures under the identification (\ref{NAH}), namely the perverse filtration associated with $\pi: \CM_{\mathrm{Dol}} \rightarrow \Lambda$ and the weight filtration with respect to the mixed Hodge structure on $\CM_B$.

The perverse filtration 
\begin{equation} \label{Perv_Filtration}
    P_0H^\ast(\CM_{\mathrm{Dol}}, \BQ) \subset P_1H^\ast(\CM_{\mathrm{Dol}}, \BQ) \subset \dots \subset P_kH^\ast(\CM_{\mathrm{Dol}}, \BQ) \subset \dots \subset H^\ast(\CM_{\mathrm{Dol}}, \BQ)
\end{equation}
is an increasing filtration defined via the perverse truncation functor \cite[Section 1.4.1]{dCHM1}. It is governed by the topology of the Hitchin system $\pi: \CM_{\mathrm{Dol}} \to \Lambda$. We recall the following useful characterization of the perverse filtration (\ref{Perv_Filtration}) by de Cataldo--Migliorini \cite{hyperplane}.

\begin{thm}[de Cataldo--Migliorini \cite{hyperplane}] \label{thm1.1}
Let $\Lambda^s \subset \Lambda$ denote an $s$-dimensional general linear sub-space. Then we have
\[
P_iH^{i+k}(\CM_{\mathrm{Dol}}, \BQ) = \mathrm{Ker}\left(H^{i+k}(\CM_{\mathrm{Dol}}, \BQ) \rightarrow H^{i+k}(\pi^{-1}(\Lambda^{k-1}), \BQ) \right).
\]
\end{thm}

\subsection{$\Gamma$-actions}\label{Section1.3}
Let $\CL \in \mathrm{Pic}^0(C)[n]$ be a $n$-torsion line bundle. Then for $(\CE, \theta) \in \CM_{\mathrm{Dol}}$, we have $(\CL \otimes \CE, \theta ) \in \CM_{\mathrm{Dol}}$. Hence the finite abelian group 
\[
\Gamma = \mathrm{Pic}^0(C)[n] \simeq (\BZ/n\BZ)^{2g}
\]
acts on $\CM_{\mathrm{Dol}}$, with the quotient
\[
\hat{\CM}_{\mathrm{Dol}} = \CM_{\mathrm{Dol}}/\Gamma
\]
a Deligne--Mumford stack parametrizing stable $\mathrm{PGL}_n$-Higgs bundles. The Hitchin map $\pi: \CM_{\mathrm{Dol}} \to \Lambda$ is $\Gamma$-equivariant with the trivial action on the Hitchin base $\Lambda$. The $\mathrm{PGL}_n$-Hitchin map $\hat{\pi}: \hat{\CM}_{\mathrm{Dol}} \to \Lambda$ fits into the commutative diagram
\begin{equation}\label{quotient}
\begin{tikzcd}[column sep=small]
\CM_{\mathrm{Dol}} \arrow{rr} \arrow[swap]{dr}{\pi}& &\hat{\CM}_{\mathrm{Dol}} \arrow{dl}{\hat{\pi}}\\
& \Lambda & 
\end{tikzcd}    
\end{equation}
where the horizontal arrow is the quotient map. We obtain from (\ref{quotient}) the canonical isomorphism
\begin{equation}\label{Gamma1}
H^*(\CM_{\mathrm{Dol}}, \BQ)^\Gamma  =  H^*(\hat{\CM}_{\mathrm{Dol}}, \BQ)
\end{equation}
compactible with the perverse filtrations,
\[
P_kH^*(\CM_{\mathrm{Dol}}, \BQ)^\Gamma  =  P_kH^*(\hat{\CM}_{\mathrm{Dol}}, \BQ),
\]
Here the perverse filtration for $\hat{\CM}_{\mathrm{Dol}}$ is associated with $\hat{\pi}: \hat{\CM}_{\mathrm{Dol}} \to \Lambda$.

We also have the corresponding $\Gamma$-action on the Betti moduli space $\CM_B$. More precisely, we view $\Gamma$ as a sub-group of $(\BC^*)^{\times 2g}$, which acts on the matrices $a_i, b_i \in \mathrm{SL}_n$ of (\ref{M_B}) by multiplication. The $\Gamma$-action on $\CM_B$ is induced by the action of the rank 1 character variety $(\BC^*)^{\times 2g}$ on the $\mathrm{GL}_n$-twisted character variety, which,  via the non-abelian Hodge correspondence,  coincides with the action of the rank 1 Hitchin moduli space $T^*\mathrm{Pic}^0(C)$ on the $\mathrm{GL}_n$-Hitchin moduli space. Hence the $\Gamma$-decomposition 
\[
H^*(\CM_B, \BQ) = H^*(\CM_B, \BQ)^\Gamma \bigoplus H_{\mathrm{var}}^*(\CM_B, \BQ)
\]
matches the $\Gamma$-decomposition (\ref{decomp}) for $\CM_{\mathrm{Dol}}$ via the non-abelian Hodge correspondence (\ref{NAH}). Analagous to (\ref{Gamma1}), we have a canonical isomorphism of mixed Hodge structures
\begin{equation}\label{Gamma2}
H^*(\CM_B, \BQ)^\Gamma  =  H^*(\hat{\CM}_B, \BQ)
\end{equation}
with $\hat{\CM}_B$ the $\mathrm{PGL}_n$-character variety diffeomorphic to $\hat{\CM}_{\mathrm{Dol}}$ via the non-abelian Hodge correspondence for  $\mathrm{PGL}_n$. 

In conclusion, we have the following proposition concerning the $P=W$ for the $\Gamma$-invariant cohomology.
\begin{prop}\label{prop1.2}
Assume that the $P=W$ conjecture (\ref{P=W}) holds for the curve $C$, the group $G= \mathrm{GL}_n$, and the degree $d$. Then we have
\[
P_k  H^*(\CM_{\mathrm{Dol}}, \BQ)^\Gamma = W_{2k}H^*(\CM_B, \BQ)^\Gamma, \quad \forall k\geq 0.
\]
\end{prop}

The following is a consequence of Proposition \ref{prop1.2} and \cite[Theorem 0.2]{dCMS}.

\begin{cor}\label{cor1.3}
When the curve $C$ has genus $g=2$, we have
\[
P_k  H^*(\CM_{\mathrm{Dol}}, \BQ)^\Gamma = W_{2k}H^*(\CM_B, \BQ)^\Gamma, \quad \forall k\geq 0.
\]
\end{cor}

\subsection{The variant cohomology}
In view of Proposition \ref{prop1.2} and Corollary \ref{cor1.3}, our main purpose of this paper is to understand the perverse filtration and the weight filtration on the variant cohomology
\[
H_{\mathrm{var}}^*(\CM_{\mathrm{Dol}}, \BQ)=H_{\mathrm{var}}^*(\CM_B, \BQ).
\]

\begin{prop}\label{prop1.4}
Let $p$ be the smallest prime divisor of $n$. We have 
\begin{equation}\label{upper}
P_{k-n(n-n/p)(g-1)}H_{\mathrm{var}}^k(\CM_{\mathrm{Dol}}, \BQ) = H_{\mathrm{var}}^k(\CM_{\mathrm{Dol}}, \BQ).
\end{equation}

\end{prop}

\begin{proof}
The argument here is a generalization of the first part of the proof of \cite[Theorem 4.4.6]{dCHM1} which treated the case $n=2$. Here we apply results of Hausel--Pauly \cite{HP} and Theorem \ref{thm1.1}. 

Let $\Lambda' \subset \Lambda$ be a general linear subspace of dimension
\begin{equation}\label{eqn9}
\mathrm{dim}(\Lambda') = n(n-n/p)(g-1) -1.
\end{equation}
Assume $\CM_{\Lambda'} = \pi^{-1}(\Lambda') \subset \CM_{\mathrm{Dol}}$. In order to prove (\ref{upper}), by Theorem \ref{thm1.1} it suffices to show 
\begin{equation}\label{111}
r\left(H_{\mathrm{var}}^k(\CM_{\mathrm{Dol}}, \BQ ) \right) = 0
\end{equation}
where $r$ is the restriction morphism
\begin{equation}\label{res}
r: H^k(\CM_{\mathrm{Dol}}, \BQ) \rightarrow H^k(\CM_{\Lambda'}, \BQ).
\end{equation}

We consider the endoscopic loci $\Lambda_{\mathrm{endo}} \subset \Lambda$ defined in \cite[Corollary 1.3]{HP}, which is formed by $a \in \Lambda$ such that the Prym variety $\mathrm{Prym}(C_a/C)$ associated with the corresponding spectral curve $C_a$ is not connected. By \cite[Lemma 7.1]{HP}, we have
\begin{equation}\label{eqn11}
\mathrm{codim}_{\Lambda}(\Lambda_{\mathrm{endo}}) = n(n-n/p)(g-1).
\end{equation}
Since $\Lambda'$ is general, it is completely contained in $\Lambda \smallsetminus \Lambda_{\mathrm{endo}}$ by (\ref{eqn9}) and (\ref{eqn11}). An identical argument as in the first paragraph of \cite[Proof of Theorem 1.4]{HP} implies that $\Gamma$ acts trivially on $H^k(\CM_{\Lambda'}, \BQ)$, \emph{i.e.},
\[
H_{\mathrm{var}}^k(\CM_{\Lambda'}, \BQ)) = 0.
\]

On the other hand, the $\Gamma$-action is fiberwise with respect to the Hitchin map $\pi: \CM_{\mathrm{Dol}} \to \Lambda$, and the restriction morphism (\ref{res}) is $\Gamma$-equivariant. In particular, we see that
\[
r\left(H_{\mathrm{var}}^k(\CM_{\mathrm{Dol}}, \BQ)\right) \subset H_{\mathrm{var}}^k(\CM_{\Lambda'}, \BQ)) = 0.
\]
This completes the proof of (\ref{111}).
\end{proof}

%(Mention that this is sharp..)

\subsection{Main results}\label{Section1.5}

The following theorem is our main result, which generalizes \cite[Theorems 4.4.6 and 4.4.7]{dCHM1} for $n=2$. It computes the perverse filtration and the weight filtration explicitly on the variant cohomology for $\mathrm{SL}_n$ with $n$ a prime number.

\begin{thm}\label{thm1.5}
Assume $n$ is a prime number, and assume
\[
c_n := n(n-1)(g-1).
\] 
\begin{enumerate}
    \item[(a)] We have
\[
0 = P_{k-c_n-1}H_{\mathrm{var}}^k(\CM_{\mathrm{Dol}}, \BQ) \subset P_{k-c_n}H_{\mathrm{var}}^k(\CM_{\mathrm{Dol}}, \BQ) = H_{\mathrm{var}}^k(\CM_{\mathrm{Dol}}, \BQ).
\]
   \item[(b)] We have
\[
0 = W_{2(k-c_n)-1}H_{\mathrm{var}}^k(\CM_{B}, \BQ) \subset W_{2(k-c_n)}H_{\mathrm{var}}^k(\CM_{B}, \BQ) = H_{\mathrm{var}}^k(\CM_{B}, \BQ).
\]
\end{enumerate}

\end{thm}

We prove Theorem \ref{thm1.5} in Section \ref{Section3}. It is clear that Theorem \ref{thm1.5} implies Theorem \ref{thm0.2}. Hence we complete the proof of Theorem \ref{thm0.1} by combining Corollary \ref{cor1.3}. 

\begin{rmk}
Proposition \ref{prop1.2} and Theorem \ref{thm1.5} combined shows that, when $n$ is a prime number, the $P=W$ conjecture for $\mathrm{SL}_n$ is equivalent to the $P=W$ conjecture for $\mathrm{GL}_n$.
\end{rmk}

For general $n$, the perverse filtration on the variant cohomology $H_{\mathrm{var}}^k(\CM_{\mathrm{Dol}}, \BQ)$ for the $\mathrm{SL}_n$-Hitchin moduli space $\CM_{\mathrm{Dol}}$ is expected to be more complicated. In view of \cite{HT}, the variant cohomology is governed by the Hitchin moduli spaces of endoscopic groups attached to irreducible non-trivial characters of $\Gamma = \mathrm{Pic}^0(C)[n]$. These endoscopic moduli spaces are further related to the $\mathrm{GL}_{n/d}$-Hitchin moduli space associated with a curve $\widetilde{C}$ given by a degree $d$ Galois cover of $C$, where $d$ runs through \emph{all} divisors of $n$. We will discuss this in a future paper.

In particular, when $n$ is prime, the relevant endoscopic Higgs bundles are of rank 1, with the corresponding moduli space the total cotangent bundle of a Prym variety. Therefore the associated perverse filtrations are trivial. This is the heuristic reason that the perverse filtrations on the variant cohomology are of the form Theorem \ref{thm1.5} (a).

\section{Decompositions of vector spaces}\label{Section2}

\subsection{$k$-sequeces} 
We consider double indexed sequences
\begin{equation}\label{dec}
\{ v_{i,j} \in \BN\}_{i,j}
\end{equation}
satisfying $v^{i,j}=0$ when $i<0$ or $j<0$. For convenience, we assume that all indices are non-negative integers.

We say that (\ref{dec}) is a \emph{$k$-sequence} if $v^{i,j} =0$ when $j\neq k$. The purpose of Section \ref{Section2} is to give two critera for $k$-sequences.

\subsection{The first criterion}
\begin{prop}\label{prop2.1}
For fixed $m,k \in \BN_{>0}$, we assume that (\ref{dec}) satisfies the following conditions:
\begin{enumerate}
    \item[(i)] $v^{i,j} =0$ if $j <k$;
    \item[(ii)] $v^{m-i,j}= v^{m+i, j}$ for any $i, j$;
    \item[(iii)] The following identify holds for any $l \geq 0$, \[
    \sum_{i+j = m+k-l} v^{i,j} = \sum_{i+j=m+k+l} v^{i,j}.
    \]
\end{enumerate}
Then (\ref{dec}) is a $k$-sequence.
\end{prop}

\begin{proof}
By (i), it suffices to show that 
\begin{equation}\label{eqn13}
v^{i,j}=0, \quad \mathrm{if}~~ k<j. %~~and~~~ i\leq m. 
\end{equation}
We prove this by induction on the value $i+j$. The induction base is the case $i+j=k$ where (\ref{eqn13}) is clearly true.

We now assume that (\ref{eqn13}) holds if $i+j<d_0$. To complete the induction, we need to show that $v^{d_0-j,j} =0$ for $k<j$. The condition (ii) implies that $v^{d_0-j,j} = v^{2m-d_0+j,j}$. On the other hand, by (iii), we have
\begin{equation}\label{eqn14}
v^{2m-d_0+j,j}+v^{2m-d_0+2j-k,k} \leq \sum_{i+j=2m-d_0+2j} v^{i,j} =  \sum_{i+j=d_0-2j+2k} v^{i,j} = v^{d_0-2j+k,k}
\end{equation}
where we apply the induction assumption in the last equation (since $d_0-2j+2k<d_0$). We deduce from (\ref{eqn14}) and (ii) that 
\[
v^{2m-d_0+j,j} \leq v^{d_0-2j+k,k} - v^{2m-d_0+2j-k,k} =0. 
\]
Hence we have $v^{d_0-j,j} = v^{2m-d_0+j,j} = 0$ which completes the induction.
\end{proof}

\subsection{The second criterion}
\begin{prop}\label{prop2.2}
For fixed $m,k \in \BN_{>0}$, we assume that (\ref{dec})  satisfies the following conditions:
\begin{enumerate}
    \item[(i)] $v^{i,j} = v^{2m+2k-i-2j,j}$ for any $i,j$.
    \item[(ii)] The following identity holds for any $l\geq 0$,
    \[
    \sum_{i+j = k+l} v^{i,j} = \sum_{j} v^{l,j}.
    \]
    \item[(iii)] The following identify holds for any $i \geq 0$, 
    \[
    \sum_{j} v^{m+i,j} = \sum_{j} v^{m-i,j}.
    \]
%    \item[(iv)] $v^{i,j} =0$ if $i+j<k$.
\end{enumerate}
Then (\ref{dec}) is a $k$-sequence.
\end{prop}

\begin{proof}
We prove that 
\begin{equation}\label{18}
v^{i,j} =0,\quad \mathrm{if}~~~k\neq j
\end{equation}
by induction on the value $i+j$. 

If $i+j\leq k$ and $j<k$, we have $v^{i,j} = v^{2m+2k-i-2j,j}$ by (i). Then (iii) implies that 
\[
v^{2m+2k-i-2j,j}\leq \sum_{l}v^{2m+2k-i-2j,l}= \sum_{l} v^{i+2j-2k,l} = 0,
\]
since $i+2j-2k<0$. Hence $v^{i,j}=0$ if $i+j<k$, and $v^{i,k-i}=0$ if $i>0$. This provides the induction base.

Now assume that (\ref{18}) holds if $i+j<d_0$. We first show that $v^{d_0-j,j}=0$ if $j>k$. In fact, by (ii) we have
\begin{equation}\label{112}
v^{d_0-j,j} + v^{d_0-j,k} \leq \sum_{j'} v^{d_0-j,j'} = \sum_{i_1+i_2 = k+(d_0-j)} v^{i_1,i_2}.
\end{equation}
Then, since $k+(d_0-j)<d_0$, the induction assumption further implies
\begin{equation}\label{113}
\sum_{i_1+i_2 = k+(d_0-j)} v^{i_1,i_2} = v^{d_0-j,k}.
\end{equation}
Combining (\ref{112}) and (\ref{113}), we have $v^{d_0-j,j} = 0$ if $j>k$. 

It remains to show that $v^{d_0-j,j}=0$ if $j<k$. In this case, we have
\[
v^{d_0-j,j} = v^{2m+2k-d_0-j,j}
\]
by (i). The condition (ii) further implies that
\begin{equation}\label{115}
v^{2m+2k-d_0-j,j}+ v^{2m+2k-d_0-j,k} \leq  \sum_{j'}v^{2m+2k-d_0-j,j'} = \sum_{i_1+i_2 = 2m+3k-d_0-j} v^{i_1,i_2}.
\end{equation}
For $i_1+i_2=2m+3k-d_0-j$, we have by (i) that $v^{i_1,i_2} = v^{j_1,j_2}$ with
\[
j_1 +j_2 = 2m+2k-(2m+3k-d_0-j) =d_0+j-k<d_0.
\]
Hence (i) and the induction assumption yield
\begin{equation}\label{116}
    \sum_{i_1+i_2 = 2m+3k-d_0-j} v^{i_1,i_2} = v^{2m+2k-d_0-j,k}.
\end{equation}
Combining (\ref{115}) and (\ref{116}), we obtain
\[
v^{d_0-j,j} = v^{2m+2k-d_0-j,j} =0
\]
which completes the induction.
\end{proof}

\section{Perverse filtrations and weight filtrations}\label{Section3}

Throughout the section, we assume that $n$ is a prime number, and complete the proof of Theorem \ref{thm1.5}. For the proof, we apply the numerical criteria of Section \ref{Section2} combined with the following ingredients:

\begin{enumerate}
    \item[(a)] Hausel-Thaddeus' topological mirror symmetry conjecture for Hitchin systems \cite{HT}, and its recent proof by Gr\"ocheneg--Wyss--Ziegler \cite{GWZ}.
    \item[(b)] The $E$-polynomials of character varieties calculated by Hausel-- Rodriguez-Villegas \cite{HRV} and Mereb \cite{SL} via point counting over finite fields.
\end{enumerate}

\subsection{The topological mirror symmetry conjecture}

Recall that the virtual Hodge polynomial $H(X; u,v)$ of an algebraic variety $X$ is
\[
H(X; t,u,v) =\sum_{i,j,k} h^{j,k}\left( \mathrm{Gr}^W_{j+k} H_c^{i}(X, \BC) \right) t^iu^jv^k
\]
where $\mathrm{Gr}_*^W$ is the graded piece with respect to the weight filtration. The $E$-polynomial of $X$ is the specialization
\[
E(X; u,v) = H(X; -1, u,v).
\]

The topological mirror symmetry conjecture proposed by Hausel--Thaddeus \cite{HT} relates the $E$-polynomial of the $\mathrm{SL}_n$-Hitchin moduli space $\CM_{\mathrm{Dol}}$ to the \emph{stringy} $E$-polynomial of the $\mathrm{PGL}_n$-Hitchin moduli space $\hat{\CM}_{\mathrm{Dol}}$. A generalized version of the Hausel--Thaddeus conjecture was proven by Gr\"ocheneg--Wyss--Ziegler \cite{GWZ} via the method of $p$-adic integrations; see also \cite{GWZ2}.

When $n$ is a prime number, we obtain the following closed formula for the $E$-polynomial of the variant cohomology of $\CM_{\mathrm{Dol}}$ from a direct calculation of the stringy $E$-polynomial of the $\mathrm{PGL}_n$-Hitchin moduli space $\hat{\CM}_{\mathrm{Dol}}$; see \cite[Proposition 8.2]{HT}.

\begin{prop}[Topological mirror symmetry \cite{HT, GWZ}]
Let $n$ be a prime number. Then we have
\begin{multline} \label{23}
   E(\CM_{\mathrm{Dol}}; u,v)-E(\hat{\CM}_{\mathrm{Dol}};u,v) = \frac{n^{2g}-1}{n}(uv)^{(n^2-1)(g-1)} \Big{(} \big{(}(u-1)(v-1)\big{)}^{(n-1)(g-1)} - \\
\Big{(}(1+u+\dots+u^{n-1})(1+v+\dots +v^{n-1})\Big{)}^{(g-1)}\Big{)}.
\end{multline}
\end{prop}

We denote $E(q)$ to be the polynomial by setting $u=v=q$ on the righthand side of (\ref{23}), 
\begin{equation}\label{E1}
E(q) := \frac{n^{2g}-1}{n} q^{\mathrm{dim}(\CM_{\mathrm{Dol}})}\left(
(q-1)^{(n-1)(2g-2)}- (1+q+\cdots+q^{n-1})^{2g-2}
\right),
\end{equation}
which is palindromic satisfying
\begin{equation}\label{E11}
   E(q) = q^{(2g-2)(2n^2+n-3)} E\left( \frac{1}{q}\right) .
\end{equation}
We denote $[E(q)]_{q^i}$ to be the coefficient of $q^i$ in the polynomial expansion of $E(q)$.

\begin{cor}\label{cor3.2}
We have
\[
\mathrm{dim}\left(H_{\mathrm{var}}^d(\CM_{\mathrm{Dol}}, \BQ)\right) = (-1)^d[E(q)]_{q^{2\mathrm{dim}\left(\CM_{\mathrm{Dol}}\right)-d}}. 
\]
\end{cor}

\begin{proof}
Since the cohomology groups $H^k$ of the moduli spaces $\CM_{\mathrm{Dol}}$ and $\hat{\CM}_{\mathrm{Dol}}$ are pure of weights $k$, their $E$-polynomials recover the virtual Hodge polynomials. Corollary \ref{cor3.2} follows from the Poincar\'e duality and (\ref{Gamma1}).
\end{proof}

\subsection{Proof of Theorem \ref{thm1.5} (a).} We define
\begin{equation}\label{decompP}
v_P^{i,j} := \mathrm{dim}\left(\mathrm{Gr}^P_iH_{\mathrm{var}}^{i+j}(\CM_\mathrm{Dol}, \BQ)\right)
\end{equation}
with $\mathrm{Gr}^P_*$ the graded piece of the perverse filtration. Recall $c_n$ from Theorem \ref{thm1.5}. It suffices to show that (\ref{decompP}) forms a $c_n$-sequence.

We check that (\ref{decompP}) satisfies (i,ii,iii) of Proposition \ref{prop2.1} for
\begin{equation}\label{km}
k =c_n, \quad m= \frac{1}{2}\mathrm{dim}(\CM_{\mathrm{Dol}}) = (n^2-1)(g-1).
\end{equation}
The condition (i) follows directly from Proposition \ref{prop1.4}. The condition (ii),
\[
\mathrm{dim}\left( \mathrm{Gr}^P_{m-i}H_{\mathrm{var}}^{m-i+j}(\CM_{\mathrm{Dol}}, \BQ) \right) = \mathrm{dim}\left( \mathrm{Gr}^P_{m+i}H_{\mathrm{var}}^{m+i+j}(\CM_{\mathrm{Dol}}, \BQ) \right),
\]
follows from the Relative Hard Lefschetz \cite{BBD} with respect to the Hitchin map $\pi: \CM_{\mathrm{Dol}} \to \Lambda$, and its compatibity with the $\Gamma$-decomposition.

Since
\[
\mathrm{dim}\left(H_{\mathrm{var}}^d(\CM_{\mathrm{Dol}}, \BQ)\right) = \sum_{i+j = d} v_{P}^{i,j},
\]
the condition (iii) is equivalent to 
\[
\mathrm{dim}\left(H_{\mathrm{var}}^{m+c_n-i}(\CM_{\mathrm{Dol}}, \BQ)\right) = \mathrm{dim}\left(H_{\mathrm{var}}^{m+c_n+i}(\CM_{\mathrm{Dol}}, \BQ)\right),
\]
which follows from Corollary \ref{cor3.2} and the symmety (\ref{E11}),
\begin{equation*}
[E(q)]_{q^i} = [E(q)]_{q^j}, \quad \mathrm{if}~~~i+j=6m-2c_n=(2n^2+n-3)(2g-2).
\end{equation*}
This completes the proof. \qed

\subsection{A symmetry}\label{Section3.3}
We see from Corollary \ref{cor3.2} that $H^2(\CM_{\mathrm{Dol}}, \BQ) = H^2(\hat{\CM}_{\mathrm{Dol}}, \BQ)$. So there is only one class $\eta$ spanning $H^2(\CM_{\mathrm{Dol}}, \BQ)$ (see \cite{Markman}), and it is relatively ample with respect to the Hitchin map. As a consequence of Theorem \ref{thm1.5} (a), we obtain the following symmetry on the cohomology of $\CM_{\mathrm{Dol}}$.

\begin{cor}\label{cor3.3}
Cupping with a power of the class $\eta$ induces an isomorphism
\begin{equation}\label{3.3}
\eta^i: H_{\mathrm{var}}^{m+c_n-i}(\CM_{\mathrm{Dol}}, \BQ ) \xrightarrow{\simeq} H_{\mathrm{var}}^{m+c_n+i}(\CM_{\mathrm{Dol}}, \BQ), \quad \forall i,m.
\end{equation}
\end{cor}

\begin{proof}
The Relative Hard Lefschetz Theorem implies that 
\begin{equation*}\label{eq25}
\eta^i: \mathrm{Gr}^P_{m-i}H_{\mathrm{var}}^{m+j-i}(\CM_{\mathrm{Dol}}, \BQ) \xrightarrow{\simeq} \mathrm{Gr}^P_{m+i}H_{\mathrm{var}}^{m+j+i}(\CM_{\mathrm{Dol}}, \BQ).
\end{equation*}
Since (\ref{decompP}) is a $c_n$-decomposition by Theorem \ref{thm1.5} (a), the only non-trivial isomorphisms (\ref{eq25}) are those with $j = c_n$, and Corollary \ref{cor3.3} follows.
\end{proof}

\begin{rmk}
In general, if $n$ is not prime, (\ref{3.3}) does not hold. In particular, Corollary \ref{cor3.3} relies heavily on the fact that (\ref{decompP}) is a $c_n$-sequence, which, by the proof of Theorem \ref{thm1.5} (a), further relies on the symmetries of the coefficients of the polynomial $E(q)$.
\end{rmk}

\subsection{$E$-polynomials of character varieties}
%For an algebraic variety, we consider a specialization
%\[
%E(X, q): = E(X; \sqrt{q}, \sqrt{q}).
%\]
%of its E-polynomial.

Recall the polynomail $E(q)$ introduced in (\ref{E1}). In view of (\ref{Gamma2}), We define the variant $E$-polynomial
\[
E_{\mathrm{var}}(\CM_B; u,v) := E(\CM_B; u,v) - E(\hat{\CM}_B; u,v).
\]
The following proposition calculates the variant $E$-polynomial for $\CM_B$. We note that the two sides of the equation (\ref{p3.5}) are of completely different flavors. The left-hand side is governed by point counting over finite fields via the character tables of $\mathrm{GL}_n(\BF_q)$ and $\mathrm{SL}_n(\BF_q)$, while the right-hand side calculates suitable cohomology groups of the moduli of certain endoscopic Higgs bundles.

\begin{prop}\label{prop3.5}
We have
\begin{equation}\label{p3.5}
E_{\mathrm{var}}(\CM_B; u,v)\cdot (uv)^{(n^2+n-2)(g-1)} = E(uv).
\end{equation}
\end{prop}

\begin{proof}
The result of Katz \cite[Appendix]{HRV} and the calculations of \cite{HRV, SL} imply that the $E$-polynomials $E(\CM_B; u,v)$ and $E(\hat{\CM}_B; u,v)$ are polynomials in the variable $q = uv$. Hence it suffices to show that 
\begin{equation}\label{Last}
     E_{\mathrm{var}}(\CM_B; \sqrt{q}, \sqrt{q}) =  \frac{n^{2g}-1}{n}q^{(n^2-n)(g-1)} \Big{(}(q-1)^{(n-1)(2g-2)} -
{(}1+q+\dots+q^{n-1}{)}^{(2g-2)}\Big{)}.
\end{equation}

By \cite[Equation (3.2.4)]{HRV} and \cite[Theorem 3.4]{SL}, we have
\begin{equation}\label{E01}
E(\hat{\CM}_B; \sqrt{q},\sqrt{q}) = \sum_{\tau}\left( q^{\frac{n^2}{2}} \frac{\CH_{\tau'}(q)}{q-1} \right)^{2g-2} C_{\tau}^0; 
\end{equation}
\begin{equation}\label{E02}
E(\CM_B; \sqrt{q},\sqrt{q}) = \sum_{\tau,t} \left( q^{\frac{n^2}{2}} \frac{\CH_{\tau'}(q)}{q-1} \right)^{2g-2} t^{2g-1}C_{\tau}^t.
\end{equation}
Here we follow the notation of \cite{HRV, SL}: the summation in (\ref{E01}) is taken over types $\tau$ of size $n$ multi-partitions and the summation in (\ref{E02}) is taken over $\tau$ of size $n$ multi-partitions and divisors $t$ of $n$ (see \cite[Section 2.5]{SL}); the polynomial $\CH_{\tau'}(q)$ is the \emph{normalized hook polynomial} associated with the conjugate $\tau'$ of the partition $\tau$ \cite[Section 3.6]{SL}; the constant $C_{\tau}^0$ is given by \cite[Equation (7)]{SL}, and \cite[Equation (33)]{SL} expresses every $C_\tau^n$ in terms of $C_\tau^0$.\footnote{See \cite[Equation (34)]{SL} for the connection between $C_{\tau}^0$ and the coefficients $C_\tau$ used in \cite{HRV}.}

Now we calculate the difference of (\ref{E01}) and (\ref{E02}).

For our purpose, we focus on 2 types of multi-partitions $\tau_1$ and $\tau_2$ as follows. Recall the type $\tau = (m_{\lambda,d})_{\lambda,d\geq 1}$ of a multi-partition from \cite[Definition 2.1]{SL}. Let $\tau_1$ be the type of the multi-partition with the only non-trivial multiplicity $m_{(1^1),1} = n$, and we calculate directly that
\begin{equation}\label{tau1}
\CH_{\tau_1'}(q) = \left(q^{-\frac{1}{2}}(1-q)\right)^n = q^{-\frac{n}{2}}(1-q)^n.
\end{equation}
Let $\tau_2$ be the type of the multi-partition with the only non-trivial multiplicity $m_{(1^1),n} = 1$, and we have
\begin{equation}\label{tau2}
\CH_{\tau_2'}(q) = q^{-\frac{n}{2}}(1-q^n).
\end{equation}
Furthermore, by a direct calculation using the concrete formula \cite[Equation (33)]{SL} for the constants $C_\tau^t$, we obtain that
\begin{enumerate}
    \item[(a)] $C_\tau^1 = C_\tau^0,~~~ C_\tau^n=0$, for $\tau \neq \tau_1, \tau_2$;
    \item[(b)] $C_{\tau_1}^1 = C_{\tau_1}^0 - \frac{1}{n}$, $C_{\tau_1}^n = 1$;
    \item[(c)] $C_{\tau_2}^1 = C_{\tau_2}^0 + \frac{1}{n}$, $C_{\tau_1}^n = -1$.
\end{enumerate}

Since $n$ is a prime number and $t$ divides $n$, the integer $t$ is either $1$ or $n$ on the right-hand side of (\ref{E02}), 
\begin{equation} \label{eq33}
E(\CM_B; \sqrt{q},\sqrt{q})= \sum_{\tau} \left( q^{\frac{n^2}{2}} \frac{\CH_{\tau'}(q)}{q-1} \right)^{2g-2} C_{\tau}^1 
+ \sum_{\tau} \left( q^{\frac{n^2}{2}} \frac{\CH_{\tau'}(q)}{q-1} \right)^{2g-2} n^{2g-1}C_{\tau}^n.
\end{equation}
By (a,b,c), (\ref{tau1}), and (\ref{E01}), we have 
\begin{multline} \label{eq34}
\sum_{\tau} \left( q^{\frac{n^2}{2}} \frac{\CH_{\tau'}(q)}{q-1} \right)^{2g-2} C_{\tau}^1 = E(\hat{\CM}_B; \sqrt{q},\sqrt{q}) \\+ \left( q^{\frac{n^2}{2}} \frac{q^{-\frac{n}{2}}(1-q)^n}{q-1} \right)^{2g-2}\cdot\left(-\frac{1}{n}\right) + 
\left( q^{\frac{n^2}{2}} \frac{q^{-\frac{n}{2}}(1-q^n)}{q-1} \right)^{2g-2}\cdot\left(\frac{1}{n}\right).
\end{multline}
Similarly, (a,b,c) and (\ref{tau2}) yield
\begin{multline} \label{eq35}
\sum_{\tau} \left( q^{\frac{n^2}{2}} \frac{\CH_{\tau'}(q)}{q-1} \right)^{2g-2}n^{2g-1} C_{\tau}^n = \\+ \left( q^{\frac{n^2}{2}} \frac{q^{-\frac{n}{2}}(1-q)^n}{q-1} \right)^{2g-2} n^{2g-1} 
+ 
\left( q^{\frac{n^2}{2}} \frac{q^{-\frac{n}{2}}(1-q^n)}{q-1} \right)^{2g-2}n^{2g-1} \cdot (-1).
\end{multline}
We complete the proof of (\ref{Last}) by combining (\ref{eq33}), (\ref{eq34}), and (\ref{eq35}).
\end{proof}

\subsection{Vanishing and Hodge--Tate}
We prove some properties of the variant cohomology of $\CM_B$ which play a crucial role in the proof of Theorem \ref{thm1.5} (b). We denote 
\begin{equation*}
    %V_W = \bigoplus_{i,j}V_W^{i.j}, \quad 
    w^{i,j}:= \mathrm{dim}\left(\mathrm{Gr}^W_i H_{\mathrm{var},c}^{j}(\CM_B, \BQ)\right)
\end{equation*}
where $H_{\mathrm{var},c}^*$ is the variant part of the compactly support cohomology.

\begin{lem}\label{lem3.6}
If $i$ is odd, or $i=2i'$ with $i'+j$ odd, we have $w^{i,j} = 0$.
\end{lem}

\begin{proof}
By Proposition \ref{prop3.5}, we have
\[
E_{\mathrm{var}}(\CM_B; q,q)= \sum_{i,j}(-1)^{j}w^{i,j}\cdot q^i =  q^{-(n^2+n-2)(2g-2)}E(q^2).
\]
In particular $w^{i,j} = 0$ if $i$ is odd. Together with Corollary \ref{cor3.2}, we have the expressions
\begin{equation}\label{eq28}
E_{\mathrm{var}}(\CM_B; q,q) = \sum_{i',j}(-1)^jw^{2i',j}q^{2i'}, \quad E(q^2) = \sum_{i',j}(-1)^j w^{2i',j}q^{2(2i'+j)}.
\end{equation}
Proposition \ref{prop3.5} further implies that
\[
\sum_{i',j}(-1)^{i'+j}w^{2i',j} = \sum_{i',j} w^{2i',j}
\]
by setting $q^2=-1$ in the equations (\ref{eq28}). Thus $w^{2i',j}=0$ if $i'+j$ is odd.
\end{proof}

The vanishing of Lemma \ref{lem3.6} implies that there is no cancellation of Hodge numbers in calculating each term of the $E$-polynomial $E_{\mathrm{var}}(\CM_B; u,v)$. In particular, we deduce the following lemma from Proposition \ref{prop3.5} that the mixed Hodge structures on the variant cohomology groups $H_{\mathrm{var},c}^{d}(\CM_B, \BQ)$ are of Hodge--Tate types.

\begin{lem}\label{lem3.7}
The mixed Hodge structure on $H_{\mathrm{var},c}^{d}(\CM_B, \BQ)$ is of Hodge--Tate type, \emph{i.e.}, 
\[
h^{i,j}(\mathrm{Gr}^W_{i+j} H_{\mathrm{var},c}^{d}(\CM_B, \BQ))=0, \quad \mathrm{if}~~i \neq j.
\]
\end{lem} 

As a corollary of Lemma \ref{lem3.7} and the Poincar\'e duality, we obtain that $H_{\mathrm{var}}^{d}(\CM_B, \BQ)$ is also of Hodge--Tate type.

\begin{cor}\label{cor3.8}
The mixed Hodge structure on $H_{\mathrm{var}}^{d}(\CM_B, \BQ)$ is of Hodge--Tate type.
\end{cor}

\subsection{Proof of Theorem \ref{thm1.5} (b).}

We use $F^\bullet H^*(X, \BC)$ to denote the Hodge filtration on the cohomology of an algebraic variety $X$. The Hodge filtration on $H^*(\CM_B, \BC)$ induces a Hodge filtration $F^\bullet H^*(\CM_B, \BC)$ on the variant cohomology.

We define the sub-vector spaces
\[
{^k\mathrm{Hdg}_{\mathrm{var}}^d}(\CM_B) : = F^kH_{\mathrm{var}}^{d}(\CM_B, \BC) \cap \bar{F}^kH_{\mathrm{var}}^d(\CM_B, \BC)  \cap W_{2k} H_{\mathrm{var}}^d(\CM_B, \BQ) \subset H_{\mathrm{var}}^d(\CM_B, \BQ)
\]
where $\bar{F}^\bullet H^*$ is the complex conjugate of the Hodge filtration. We obtain from Corollary \ref{cor3.8} that
\begin{equation}\label{321}
\mathrm{dim}\left( {^k\mathrm{Hdg}_{\mathrm{var}}^d}(\CM_B) \right) =\mathrm{dim} \left( \mathrm{Gr}^W_{2k} H^d(\CM_B, \BQ) \right).
\end{equation}

Recall the class $\eta \in H^2(\CM_B, \BQ)$ introduced in Section \ref{Section3.3}, which lies in ${^2\mathrm{Hdg}_{\mathrm{var}}^2}(\CM_B)$ by \cite{Shende}. Hence, Corollary \ref{cor3.3} implies that cupping with $\eta^i$ induces an isomorphism
\begin{equation}\label{233}
\eta^i: {^r\mathrm{Hdg}_{\mathrm{var}}^{r+c_n-i}}(\CM_B) \xrightarrow{\simeq}  {^{r+2i}\mathrm{Hdg}_{\mathrm{var}}^{r+c_n+i}}(\CM_B), \quad \forall r \in \BN.
\end{equation}

Now we consider 
\[
v_W^{i,j}:= \mathrm{dim}\left({^i\mathrm{Hdg}_{\mathrm{var}}^{i+j}}(\CM_B)\right).
\]
In view of (\ref{321}), it suffices to check that $\{v_W^{i,j}\}_{i,j}$ satisfies (i,ii,iii) of Proposition \ref{prop2.2} with $k$ and $m$ given by (\ref{km}).

The condition (i) follows from (\ref{233}). Next, we verify the condition (iii). By Lemma \ref{lem3.6} and Proposition \ref{prop3.5}, each summation $\sum_j{v_W^{i,j}}$ is given by a coefficient of the polynomial $E(q)$, and the condition (iii) follows from the symmetry (\ref{E11}).

Finally, we obtain from Proposition \ref{prop3.5} and the equation (\ref{E11}) that
\begin{equation}\label{23333}
E\left(\frac{1}{q}\right) q^{2\mathrm{dim}(\CM_B)} = E_{\mathrm{var}}(\CM_B; \sqrt{q}, \sqrt{q})q^{c_n}.
\end{equation}
By Corollary \ref{cor3.2}, the left-hand side of (\ref{23333}) computes \[
\mathrm{dim}\left(H_{\mathrm{var}}^d(\CM_{\mathrm{Dol}}, \BQ)\right) = \sum_{i+j=d} v_W^{i,j},
\]
while the right-hand side computes $\sum_{j} v_W^{d-c_n,j}$ by the definition of $E$-polynomials and the vanishing of Lemma \ref{lem3.6}. Hence the condition (ii) holds. This completes the proof.\qed

\section{Hitchin moduli spaces and compact hyper-K\"ahler manifolds} \label{Section4}

\subsection{Overview}

A crucial step in the proof of the $P=W$ conjecture for genus 2 and $\mathrm{GL}_n$ in \cite{dCMS} is to use degenerations connecting certain compact hyper--K\"ahler manifolds and Hitchin moudli spaces. More precisely, we embedd a genus 2 curve $C$ into an abelian surface $A$,
\begin{equation}\label{j}
j: C \hookrightarrow A.
\end{equation}
The degeneration to the normal cone associated with (\ref{j}) yields a flat family 
\begin{equation}\label{eqn16}
\CM \to  \mathbb{A}^1.
\end{equation}
Its general fiber is a compact (non-simply connected) hyper--K\"ahler manifold $\CM_{n[C],A}$ which is the moduli of certain stable 1-dimensional sheaves supported on the curve class
\[ 
n[C] \in H_2(A, \BZ),
\]
and its central fiber is the $\mathrm{GL}_n$-Hitchin moduli space $\CM_{\mathrm{Dol}}^{\mathrm{GL}_n}$. See  \cite{DEL} and \cite[Section 4.2]{dCMS} for more details about this degeneration.

We construct in \cite[Section 4.3]{dCMS} a \emph{surjective} specialization morphism 
\begin{equation}\label{sp}
\mathrm{sp}^!: H^*(\CM_{n[C],A}, \BQ) \rightarrow H^*(\CM_{\mathrm{Dol}}^{\mathrm{GL}_n}, \BQ) 
\end{equation}
which is a morphism of $\BQ$-algebras preserving the perverse filtrations and tautological classes constructed from universal families. Hence the morphism (\ref{sp}) governs the tautological generators in $H^*(\CM_{\mathrm{Dol}}^{\mathrm{GL}_n}, \BQ)$.

A degeneration similar to $\CM \to \mathbb{A}^1$ can also be constructed for the $\mathrm{SL}_n$-Hitchin moduli space $\CM_{\mathrm{Dol}}$. More precisely, under the degeneration (\ref{eqn16}), the albenese map (see \cite{Yo})
\[
\CM_{n[C],A} \rightarrow \mathrm{Pic}^d(A) \times A
\]
degenerates to the morphism
\[
\mathrm{det}\times \mathrm{trace}: \CM_{\mathrm{Dol}}^{\mathrm{GL}_n} \rightarrow \mathrm{Pic}^d(C) \times \BA^2.
\]
By taking fibers, we obtain a flat family $\CM^{\mathrm{SL}} \to \BA^1$ with general fiber $\CK_{n[C],A}$ an irreducible hyper--K\"ahler manifold of Kummer type\footnote{We say that a hyper-K\"ahler manifold is of Kummer type if it deforms to a generalized Kummer variety.} and central fiber the $\mathrm{SL}_n$-Hitchin moduli space $\CM_{\mathrm{Dol}}$. Moreover, the variety $\CK_{n[C],A}$ admits a Lagrangian fibration 
\[
\CM_{\mathrm{Dol}} \rightarrow \BP^N= |nC|
\]
degenerating to the Hitchin map $\pi: \CM_{\mathrm{Dol}} \rightarrow \Lambda$.\footnote{Since this construction is not essentially used in the present paper, we omit further details.} By the construction in \cite[Section 4.3]{dCMS}, this yields a specialization morphism
\begin{equation}\label{sp2}
    \mathrm{sp}^!: H^*(\CK_{n[C],A}, \BQ) \to H^*(\CM_{\mathrm{Dol}}, \BQ)
\end{equation}
preserving the perverse filtrations. It is natural to ask whether (\ref{sp2}) is surjective. More general, we are interested in exploring whether the cohomology of $H^*(\CM_{\mathrm{Dol}}, \BQ)$ can be governed by the cohomology of a compact irreducible hyper-K\"ahler manifold, so that we can extend the method of \cite{dCMS} to studying the perverse filtration for the $\mathrm{SL}_n$-Hitchin system $\pi: \CM_{\mathrm{Dol}}\to \Lambda$.

\begin{question}
Does there exist a grading preserved surjective morphism 
\begin{equation}\label{eqn18}
f: H^*(M, \BQ) \rightarrow H^*(\CM_{\mathrm{Dol}}, \BQ)
\end{equation}
of graded $\BQ$-algebras such that $M$ is a compact irreducibel hyper-K\"ahler manifold?
\end{question}

In this Section, we discuss obstructions to the existence of (\ref{eqn18}).

\subsection{An obstruction for $\mathrm{SL}_2$}
From now on, let $\CM_{\mathrm{Dol}}$ be the moduli space of stable Higgs bundles attached to a genus 2 curve $C$, the group $\mathrm{SL}_2$, and a degree 1 line bundle $L \in \mathrm{Pic}^1(C)$; see Section \ref{Section1}. The variety $\CM_{\mathrm{Dol}}$ is nonsingular of dimension $6$.

The following proposition provides a necessary condition for the cohomology of $\CM_{\mathrm{Dol}}$ to be governed by the cohomology of another manifold $M$.

\begin{prop}\label{prop4.2}
Assume $M$ is a manifold with a grading preserved surjective morphism
\begin{equation}\label{eqn20}
f: H^*(M, \BQ) \rightarrow H^*(\CM_{\mathrm{Dol}}, \BQ)
\end{equation}
of graded $\BQ$-algebras. Then we have
\begin{equation}\label{dim}
\mathrm{dim} \left( H^5(M, \BQ)/\big{[}( H^2(M, \BQ) \cup H^3(M, \BQ)\big{]} \right) \geq 30.
\end{equation}
\end{prop}

\begin{proof}
Assume that (\ref{eqn20}) is surjective. Recall the decomposition (\ref{decomp}). By \cite{Hit}, we have
\begin{equation}\label{123}
H_{\mathrm{var}}^*(\CM_{\mathrm{Dol}}, \BQ) = H_{\mathrm{var}}^5(\CM_{\mathrm{Dol}}, \BQ) ,\quad  \mathrm{dim}\left(H_{\mathrm{var}}^5(\CM_{\mathrm{Dol}}, \BQ)\right) =30.
\end{equation}
Since $H^2(\CM_{\mathrm{Dol}}, \BQ)$ and $H^3(\CM_{\mathrm{Dol}}, \BQ)$ lie in the invariant part $H^*(\CM_{\mathrm{Dol}}, \BQ)^\Gamma$ and $f$ is grading preserved, we have
\[
f\big{[}( H^2(M, \BQ) \cup H^3(M, \BQ)\big{]} \subset H^5(\CM_{\mathrm{Dol}}, \BQ)^\Gamma. 
\]
Hence we obtain a surjective morphism
\[
H^5(M, \BQ)/\big{[}( H^2(M, \BQ) \cup H^3(M, \BQ)\big{]} \rightarrow H^5_{\mathrm{var}}(\CM_{\mathrm{Dol}},\BQ)
\]
which implies (\ref{dim}).
\end{proof}

\subsection{Compact hyper-Ka\"ahler manifolds}
Recall that all known examples of compact irreducible hyper-K\"ahler manifolds belong to the following families:
\begin{enumerate}
    \item[(a)] The $K3$ type and the Kummer type \cite{B};
    \item[(b)] O'Grady's 6-dimensional family (OG6 type) \cite{OG6};
    \item[(c)] O'Grady's 10-dimensional family (OG10 type) \cite{OG10}.
\end{enumerate}

Combining with structural results of the cohomology of hyper--K\"ahler manifolds \cite{LL, Ver95, Ver96}, Proposition \ref{prop4.2} implies that $M$ cannot be one of the known examples listed above of irreducible hyper--K\"ahler 6-folds for a surjective morphism (\ref{eqn20}) to exist. 

\begin{prop}\label{prop4.3}
Assume $M$ is a hyper-K\"ahler 6-fold of $K3$, Kummer, or OG6 type, then any grading preserved morphism 
\[
f: H^*(M, \BQ) \rightarrow H^*(\CM_{\mathrm{Dol}}, \BQ)
\]
is not surjective. In particular, the specialization morphism (\ref{sp2}) is not a surjection.
\end{prop}

\begin{proof}
By (\ref{123}), the variety $\CM_{\mathrm{Dol}}$ has non-trivial odd cohomology. Therefore the calculations of \cite{Go1, MRS} imply that $M$ is not of $K3$ or OG6 type whose odd cohomology vanishes.

The cohomology of a manifold of Kummer type admits an action of the Looijenge--Lunts--Verbitsky (LLV) Lie algebra $\mathfrak{so}(4,5)$; see \cite{LL, Ver95, Ver96}. If $M$ is 6-dimensional, the precise form of the LLV decomposition of $H^*(M, \BR)$ with respect to $\mathfrak{so}(4,5)$-representations was calculated in \cite[Corollary 3.6]{GKLR}. In particular, the odd cohomology $H^{\mathrm{odd}}(M, \BR)$ is an irreducible $\mathfrak{so}(4,5)$-module whose highest weight vector lying in $H^3(M, \BR)$. Therefore we obtain that
\[
 H^2(M, \BR) \cup H^3(M, \BR) = H^5(M, \BR).
\]
This contradicts Proposition \ref{prop4.2}.
\end{proof}

\end{document}